\documentclass[preprint,draft,letter]{elsarticle}

\usepackage[all,poly]{xy}
\usepackage{amsfonts,amsmath,amssymb}
\usepackage{multicol}
\usepackage{multirow}
\usepackage{color}

\newtheorem{thm}{Theorem}[section]
\newtheorem{prop}[thm]{Proposition}
\newtheorem{lem}[thm]{Lemma}
\newtheorem{cor}[thm]{Corollary}

\newdefinition{defn}[thm]{Definition}
\newdefinition{remark}{Remark}
\newtheorem{exmp}[thm]{Example}
\newproof{proof}{Proof}

\begin{document}

\title{Complex Equiangular Tight Frames and Erasures}

\author{Thomas R. Hoffman}
\ead{thoffman@coastal.edu}
\author{James P. Solazzo}
\ead{jsolazzo@coastal.edu}
\address{Coastal Carolina University, Conway, SC}


\begin{abstract}
In this paper we  demonstrate that there are  distinct differences between 
real and complex equiangular tight frames (ETFs) with regards to erasures. 
For example,  we prove that there 
exist arbitrarily large non-trivial complex equiangular tight frames which are 
robust against three erasures,
and that such frames come from a unique class of complex ETFs. In addition, we
extend certain results in \cite{BP} to complex vector spaces as well as show that other  
results regarding real ETFs are not valid for complex ETFs.

\end{abstract}

\maketitle

%*********************************************************************************************************
\section{Introduction}
%*********************************************************************************************************

In the last few years the search for equiangular tight frames (ETFs) has become 
increasingly popular
\cite{BP,BPT,BW,DHS,HP,Kal,SH,STDH}. The main reason for this increased interest is that 
ETFs minimize the ``error'' for two erasures in certain communication networks \cite
{HP,SH}.

In this paper we extend a result in \cite{BP} regarding real
ETFs to complex ETFs. We also demonstrate that there are  distinct differences between 
real and complex ETFs. 
For example, the  real $3$-uniform frames correspond
precisely to the so-called trivial real ETFs \cite{BP}. However, we prove that there 
exist arbitrarily large non-trivial complex $3$-uniform
frames, and that such frames come from a unique class of complex ETFs. Consequently, 
there exist complex ETFs which are also robust against three 
erasures. Furthermore, we show that there exist only one class of ETFs robust against
four erasures, and in some sense this class is ``trivial''. 

The paper is organized as 
follows: Section 2 outlines the relationship between equiangular tight frames 
and a certain class of matrices called {\it Seidel matrices}, and Section 3 includes the 
results and examples. Readers familiar with the work in \cite{HP,SH,BP} may go straight
to Section 3.

%*********************************************************************************************************
\section{Preliminaries}
%*********************************************************************************************************

It is assumed that the reader is familiar with the basic definitions and theorems of frame theory. Both of the papers \cite{Cas,KC} are recommended as an introduction to the general theory on frames. For a detailed discussion on ETFs, and much of the motivation behind this paper, the authors further recommend reading \cite{BP,HP}.

The following definition and theorem are due to Holmes and Paulsen \cite{HP}.

\begin{defn}
An $n \times n$ self-adjoint matrix $Q$ satisfying $q_{ii}=0$ and
$|q_{ij}|=1$ for all $i \not= j$ is called a {\bf Seidel matrix}.
\end{defn}

Note that some authors refer to a Seidel matrix as a {\it signature matrix}.

\begin{remark}
When $Q$ is a real Seidel matrix, $A=1/2(Q-I+J)$ is the adjacency matrix for a graph. We consider this graph as associated to the frame corresponding to $Q$.
\end{remark}

\begin{thm}[Theorem 3.3 of \cite{HP}]\label{rod}
Let $Q$ be a Seidel matrix. Then the following are equivalent:
\begin{enumerate}
\item $Q$ is the Seidel matrix of an equiangular tight frame,
\item $Q^2=(n-1)I+\mu Q$ for some necessarily real number $\mu$,
\item $Q$ has exactly two eigenvalues.
\end{enumerate}
\end{thm}

Note that 
condition (2) in Theorem \ref{rod} is particularly useful for the computational aspects of constructing a Seidel matrix
$Q$ associated with an equiangular tight frame. Furthermore,
a Seidel matrix $Q$ satisfying any of the three equivalent conditions in Theorem \ref{rod} yields several useful parameters. It is
shown in \cite{HP}, if $\lambda_1 < 0 < \lambda_2$ are $Q$'s two eigenvalues, then the parameters $n, \ k, \ \mu, \ \lambda_1,$ and $\lambda_2$
satisfy the following properties:
\begin{eqnarray}\label{eqn:kfromnmu}
\mu=(n-2k)\sqrt{\frac{n-1}{k(n-k)}} = \lambda_1+ \lambda_2, \qquad k=\frac{n}{2}-
\frac{\mu n}{2\sqrt{4(n-1)+\mu^2}} \\ \nonumber
\lambda_1= -\sqrt{\frac{k(n-1)}{n-k}},  \qquad \lambda_2 = \sqrt{\frac{(n-1)(n-k)}{k}}, \quad n=1-\lambda_1\lambda_2.
\end{eqnarray}

In order to better understand this relationship between a Seidel matrix Q with
two distinct eigenvalues and its associated equiangular tight frame, we need  the following theorem about finite 
dimensional frames.

\begin{thm} Let $\mathbb{F}$ denote the 
field of real or complex numbers. The family $F=\{ x_i \}_{i=1}^n \subset \mathbb{F}^k$ is 
a Parseval frame for $\mathbb{F}^k$ if and only if the analysis operator $V$ associated with $F$ is an isometry.
\end{thm}

For the remainder of this paper it will be assumed that a frame $F$ is a Parseval frame. We
will refer to a Parseval frame with $n$ vectors in $\mathbb{F}^k$ as an {\bf (n,k)-frame}.
If we consider an element $x$ in $\mathbb{F}^k$ as a column vector, then the rows of the analysis operator $V$ are the adjoints
of the frame vectors in $F$. 

In \cite{HP,BP}, they discuss one way to
consider a frame as a code. It is their idea that is the main impetus for the work in this paper. We end this section by outlining this idea, and then show how Seidel matrices arise in the study of $(n,k)$-frames.

Given a vector $x$ in $\mathbb{F}^k$ and an $(n,k)$-frame with analysis operator $V$, 
consider the vector
$Vx$ in $\mathbb{F}^n$ as an encoded version of $x$, and simply decode $Vx$ by applying $V^*$.
Let $E$ denote the diagonal matrix of $m$ zeros and $n-m$ ones. Thus the vector $EVx$ is 
just the vector $Vx$ with $m$-components erased corresponding to the zeros in the diagonal entries of $E$. It is said
that {\it m-erasures} have occurred during transmission.
One way to decode the received vector $EVx$ with $m$ erasures is to again apply $V^*$. The error in reconstructing $x$ by multiplying $EVx$ on the left by $V^*$ is given by 
$$
\| x-V^*EV \| = \| V^*(I-E)Vx \| = \| V^*DVx \|
$$
where $D$ is the diagonal matrix of $m$ ones and $n-m$ zeros. The operator
$V^*DV$ is referred to as the {\it error operator}. This is only one of several methods possible for reconstructing $x$. However, it is this particular method which led Bodmann and Paulsen in \cite{BP} to introduce the following definition. The quantity in Definition \ref{inftynormdef} represents the maximal norm of an error operator given that some set of $m$ erasures occurs. 

\begin{defn} \label{inftynormdef}
Let $\mathcal{D}_m$ denote the set of diagonal matrices that have
exactly $m$ diagonal entries equal to one and $n-m$ entries equal to
zero.  Given an $(n,k)$-frame $F$, set
$$
e_m^{\infty}(F):=\mathrm{max}\{ \| V^*DV \|  : D \in \mathcal{D}_m \},
$$
where $V$ is the analysis operator of $F$, and the norm of the matrix is understood to be the operator norm.
\end{defn}

An $(n,k)$-frame  $F$ in $\mathbb{F}^k$ where $\|f_i\|$ is a constant for each $i=1,...,n$ is
commonly referred to as an {\it equal norm frame} in the current literature. If $F$ has the additional property that
$| \langle f_j,f_i \rangle |$ is a constant whenever $i \not= j$, then $F$ is an ETF.
For the purposes of this paper, we will refer to {\it equal norm frames} as {\bf uniform frames} and to {\it 
equiangular tight frames} as {\bf 2-uniform frames} as in \cite{HP}. Both uniform and $2$-uniform frames are important
since an $(n,k)$-frame $F$ minimizes
the quantity $e_1^{\infty}(F)$ if and only if $F$ is a uniform frame \cite{CK}, and minimizes the quantity 
$e_2^{\infty}(F)$ if and only if $F$ is a $2$-uniform frame
\cite{HP}. Indeed, uniform and $2$-uniform frames are robust against one and two erasures respectively. In Section 3 we
show that there exist a class of complex $(n,k)$-frames which are robust against one, two and three erasures. Furthermore, we show that this is the only such class of frames.

We end this section by showing the connection between $2$-uniform frames (or ETFs) and Seidel matrices. Suppose
$F$ is a $2$-uniform frame, and that $V$ is the associated analysis operator for $F$. It is easy to show that an
$(n,k)$-frame $F$ is a uniform frame if and only if $\|f_i\|=\sqrt{\frac{k}{n}}$ for each $i=1,...,n$. It follows that
$$
VV^*=\frac{k}{n}I+c_{n,k}Q
$$
where $Q$ is a Seidel matrix and $c_{n,k}=\sqrt{\frac{k(n-k)}{n^2(n-1)}}=| \langle f_j,f_i \rangle |$. It is worth noting that each
$n\times n$  Seidel matrix produces a set of equiangular lines in $\mathbb{C}^k$ for some $k < n$. However, the
vectors corresponding to this set of equiangular lines do not necessarily span $\mathbb{C}^k$, and consequently they may
not form a frame for $\mathbb{C}^k$. This is precisely why Theorem \ref{rod} is important, it provides necessary and sufficient
conditions in order for a Seidel matrix (real or complex) to produce a $2$-uniform frame (ETF).

%%%%%%%%%%%%%%%%%%%%%%%%%%%%%%%%%%%%%%%%%%%%%%%%%%%%%%%%%%%%%%%%%%%%%%%%%%%%%%%%%%%%%%%%%%%%%%%%%%%%%%%%%%%%%%%%%%%%%%%%%%%%%%
\section{Results}
%%%%%%%%%%%%%%%%%%%%%%%%%%%%%%%%%%%%%%%%%%%%%%%%%%%%%%%%%%%%%%%%%%%%%%%%%%%%%%%%%%%%%%%%%%%%%%%%%%%%%%%%%%%%%%%%%%%%%%%%%%%%%%

\begin{defn} Let $F$ be an $(n,k)$-frame in $\mathbb{F}^k$. We will call $F$ an {\bf m-uniform frame} provided that 
$\| V^*DV \|$ is a constant for each $D$ in $\mathcal{D}_m$. $F$ is called a {\bf completely m-uniform frame}, denoted {\bf $m_c$-uniform frame}, if
$F$ is an $\ell$-uniform frame for each $\ell=1,...,m$.
\end{defn}

Note, that there is a distinction between a {\it $2$-uniform frame} $F$ in the above definition 
and what the authors in \cite{HP,BP} refer to as a {\it $2$-uniform frame}. Namely, that a $2$-uniform frame in \cite{HP,BP} is what
we refer to as a $2_c$-uniform frame.

Along with introducing the error operator stated in Definition \ref{inftynormdef}, the authors in \cite{BP} developed error estimates of this operator. Their key result for these estimates is Theorem 5.3, which we restate here.

\begin{thm}[Theorem 5.3 of \cite{BP}] \label{Holmes}
Let $F$ be a real $2$-uniform $(n,k)$-frame. Then $e_m^{\infty}\leq k/n+(m-1)c_{n,k}$ with equality if and only if a graph associated with $F$ contains an induced subgraph on $m$ vertices that is complete bipartite.
\end{thm}

The following proposition summarizes results about real $3_c$-uniform frames which follow as corollaries to Theorem 5.3 in \cite{BP}.

\begin{prop}\label{threecorollaries}
Let $F$ and $G$ be real $2_c$-uniform (n,k)-frames.
\begin{enumerate}
\item The graph associated with $F$ either contains an induced complete bipartite graph on $3$ vertices or it
is switching equivalent to the complete graph on $n$ vertices. Consequently, if $k < n-1$, then $e_3^{\infty}(F)= \frac{k}{n} +2c_{n,k}$.
\item $e_3^{\infty}(F)=e_3^{\infty}(G).$
\item The trivial $2_c$-uniform (n,k)-frames, corresponding to $k=1$ and $k=n-1$, are $3$-uniform. Conversely, if $F$ is 
a real $3_c$-uniform (n,k)-frame, then either $k=1$ or $k=n-1$ and it is equivalent to the corresponding trivial frame.
\end{enumerate}
\end{prop}

In \cite{BP}, the authors used the connection between real Seidel matrices and graphs to prove Theorem \ref{Holmes} and the results listed in Proposition \ref{threecorollaries}. They extended this connection to Seidel matrices containing third roots of unity using directed graphs in \cite{BPT}. Unfortunately, there is no obvious extension of this idea to connect arbitrary complex Seidel matrices with a currently known class of graphs. However, the fact that these known proofs of the real do not extend to the complex case does not mean these statements do not hold. In particular, we are able to recover an analog for Theorem 5.3 and one of its corollaries. Furthermore, we provide counterexamples for the other two corollaries.

The following proposition is the key ingredient
in determining an upper bound for $e_m^{\infty}(F)$ as well as when the upper bound is saturated. The real case of 
Proposition \ref{largesteigenvalue} below is part of the proof of Theorem 5.3 in \cite{BP}.

\begin{prop} \label{largesteigenvalue}
If $Q$ is a Seidel adjacency matrix , then $\|Q \|$ is at most $n-1$. Moreover, $\| Q \|=n-1$ if and only
if $Q = J-I$.
\end{prop}

\begin{proof}
First note that the largest eigenvalue of $J_n$, the matrix of all ones, is $n$. For any vector $x$ in $\mathbb{C}^n$ and any $S$, taking the moduli
of all their entries can only increase
the value of the expression
\begin{eqnarray} \label{increasing}
\frac{|\langle (I_n+Q)x, x \rangle |}{\|x\|^2}.
\end{eqnarray}
Since $I_n+Q$ is a Hermitian matrix $\|I_n+Q\|$ is the maximum of the moduli of the eigenvalues of $I_n+Q$. Let $x$ be an eigenvector of $I_n+Q$ corresponding to the eigenvalue, $\lambda$, of largest modulus, and let $\overline{x}=(|x_1|,...,|x_n|)$. It follows that:
\begin{equation*}
\| I_n +  Q \|  =  | \lambda |
           =  \frac{ | \langle (I_n+Q)x, x \rangle |}{\|x\|^2}
          \leq  \frac{ | \langle J_n\overline{x}, \overline{x} \rangle |}{\|x\|^2}
          \leq  \frac{ \|J_n\overline{x}\| \|x\|}{\|x\|^2}
          \leq n.
\end{equation*}
Hence, $\|Q\|$ is at most $n-1$. 

Since $n$ is the largest eigenvalue of matrix $J$ it follows that $n-1$ is largest eigenvalue of $Q=J-I$. In addition, the expression
in (\ref{increasing}) can only increase. Thus, if $Q$ does not equal $J-I$, then $\| Q \| < n-1$.
\end{proof}
 
Let $Q_m$ denote a compression of $Q$ to $m$ rows and $m$ columns. We say two Seidel matrices $Q$ and $S$ are 
{\it switching equivalent} if there exists a permutation matrix $P$ and a diagonal matrix $D$ whose diagonal entries have modulus $1$ such that $Q=PDSD^{-1}P^{-1}$.

\begin{cor}\label{upperbound}
Let $F$ be $2_c$-uniform (n,k)-frame (real or complex) and let $Q$ be the associated Seidel matrix of the corresponding projection
$VV^*$. Then
\begin{eqnarray}\label{merasure}
\mathrm{e}_m^{\infty}(F) \leq \frac{k}{n}+(m-1)c_{n,k}
\end{eqnarray}
with equality if and only if there is a $Q_m$ switching equivalent to $J_m-I_m$.
\end{cor}

\begin{proof} Let $F$ be an equiangular (n,k)-frame and $V$ be the corresponding
analysis operator for $F$. Since $VV^*$ is a positive operator, its compression
$(VV^*)_m$, where $1 \leq m \leq n$, to the rows and columns where $D$ has $1$'s, is
also a positive operator. Thus, determining the norm of $\|V^*DV\|=\| DVV^*D \|$ is equivalent
to finding the largest eigenvalue of $(VV^*)_m$. We further reduce this problem to finding the largest eigenvalue of $Q_m$, where $(VV^*)_m=\frac{k}{n}I+c_{n,k}Q_m.$ By
Proposition \ref{largesteigenvalue},
$$
\| DVV^*D \|=\|D(\frac{k}{n}I_n+c_{n,k}Q)D\|=
\|\frac{k}{n}I_m +c_{n,k}Q_m\| \leq \frac{k}{n}+(m-1)c_{n,k}
$$
with equality if and only
if $Q = J-I$.
\end{proof}

\begin{remark} The above corollary is the complex version of Theorem 5.3 in \cite{BP}. 
In \cite{BP}, $2_c$-uniform (n,k)-frames for which
$\| DVV^*D \| $ is a constant for every $D$ in $\mathcal{D}_3$ are called {\it 3-uniform }, or in the terminology of this paper $3_c$-uniform.
\end{remark}

Example \ref{complexcounter} below shows that there are complex $2_c$-uniform $(n,k)$-frames, say $F$ and $G$, for 
which $e_3^{\infty}(F) \not= e_3^{\infty}(G)$ which violates parts (1) and (2) of Proposition \ref{threecorollaries}.

\begin{exmp}\label{complexcounter}
Let $F$ and $G$ be the complex $2_c$-uniform $(9,3)$-frames corresponding to the Seidel matrices

$$\begin{bmatrix} 
0&1&1&1&1&1&1&1&1\\ 
1&0&-1&\omega^{5}&\omega^{5}&\omega^{5}&\omega&\omega&\omega\\ 
1&-1&0&\omega&\omega&\omega&\omega^{5}&\omega^{5}&\omega^{5}\\ 
1&\omega&\omega^{5}&0&\omega^{5}&\omega&-1&\omega^{5}&\omega\\ 
1&\omega&\omega^{5}&\omega&0&\omega^{5}&\omega^{5}&\omega&-1\\ 
1&\omega&\omega^{5}&\omega^{5}&\omega&0&\omega&-1&\omega^{5}\\ 
1&\omega^{5}&\omega&-1&\omega&\omega^{5}&0&\omega&\omega^{5}\\ 
1&\omega^{5}&\omega&\omega&\omega^{5}&-1&\omega^{5}&0&\omega\\ 
1&\omega^{5}&\omega&\omega^{5}&-1&\omega&\omega&\omega^{5}&0\\ 
\end{bmatrix}$$ and

$$\begin{bmatrix} 
0&1&1&1&1&1&1&1&1\\ 
1&0&-1&\omega^{5}&\omega^{5}&\omega^{5}&\omega&\omega&\omega\\ 
1&-1&0&\omega&\omega&\omega&\omega^{5}&\omega^{5}&\omega^{5}\\ 
1&\omega&\omega^{5}&0&\omega^{5}&\omega&1&\omega^{4}&\omega^{2}\\ 
1&\omega&\omega^{5}&\omega&0&\omega^{5}&\omega^{2}&1&\omega^{4}\\ 
1&\omega&\omega^{5}&\omega^{5}&\omega&0&\omega^{4}&\omega^{2}&1\\ 
1&\omega^{5}&\omega&1&\omega^{4}&\omega^{2}&0&\omega^{5}&\omega\\ 
1&\omega^{5}&\omega&\omega^{2}&1&\omega^{4}&\omega&0&\omega^{5}\\ 
1&\omega^{5}&\omega&\omega^{4}&\omega^{2}&1&\omega^{5}&\omega&0\\ 
\end{bmatrix}$$
respectively, where $\omega$ is a primitive $6^{th}$ root of unity. By computation, we get $e_3^\infty(F)\approx.6465$ which is strictly less than $\frac{k}{n}+2c_{n,k}=\frac{2}{3}$ disproving part (1) of Proposition \ref{threecorollaries}. Since $e_3^\infty(G)\approx\frac{2}{3}$, we also see that part (2) of Proposition \ref{threecorollaries} fails to hold for complex matrices.
\end{exmp}

Part (3) of Proposition \ref{threecorollaries} states
that the only real $3_c$-uniform (n,k)-frames are the trivial $(n,k)$-frames. However, the following example shows that in the complex case there 
exist non-trivial $3_c$-uniform frames.

\begin{exmp}\label{3unif}
Let $F$ and $G$ be the complex $2_c$-uniform frames corresponding to the Seidel matrices

$$
\begin{bmatrix}
0&1&1&1\\ 
1&0&-i&i\\ 
1&i&0&-i\\ 
1&-i&i&0\\
\end{bmatrix}
$$
and
$$
\begin{bmatrix}
0&1&1&1&1&1&1&1\\ 
1&0&-i&-i&-i&i&i&i\\ 
1&i&0&-i&i&-i&-i&i\\ 
1&i&i&0&-i&-i&i&-i\\ 
1&i&-i&i&0&i&-i&-i\\ 
1&-i&i&i&-i&0&-i&i\\ 
1&-i&i&-i&i&i&0&-i\\ 
1&-i&-i&i&i&-i&i&0\\ 
\end{bmatrix}
$$ respectively. These frames are both $3_c$-uniform and neither of them is a trivial $(n,1)$ or $(n,n-1)$-frame.
\end{exmp}

The $2_c$-uniform frames corresponding to the Seidel matrices in Example \ref{3unif} come from real skew-symmetric matrices with two distinct eigenvalues. A more detailed
discussion of $2_c$-uniform frames which arise from such matrices can be found in \cite{DHS}. The following theorem shows that all $2_c$-uniform frames which arise from a real skew-symmetric matrix with two distinct eigenvalues are $3_c$-uniform.

\begin{thm} \label{bingo} 
Let $A$ be real skew-symmetric matrix with two distinct eigenvalues and entries $a_{i,j}=\pm 1$ when $i\ne j$ and $0$ otherwise. The frame corresponding to the Seidel matrix $Q=iA$ is $3_c$-uniform.
\end{thm}

\begin{proof}
By Proposition 3.1 of \cite{DHS}, the standard form of $Q$ has entries
$$q_{i,j}=
\begin{cases}
\pm i, &\text{if $1<i$, $1<j$, and $i\ne j$;}\\
0, &\text{if $i=j$;}\\
1, &\text{otherwise}
\end{cases}
$$
Thus, every compression of $Q$ to three rows and three columns is either of
the form
$$
\begin{bmatrix}
0 & 1 & 1 \\
1 & 0 & i \\
1 & -i &0
\end{bmatrix} 
\text{ or }
\begin{bmatrix}
0 & 1 & 1 \\
1 & 0 & -i \\
1 & i &0
\end{bmatrix}
$$
Consequently, $\| V^*DV\|$ is a constant for all $D$ in $\mathcal{D}_3$ from which the 
result follows.
\end{proof}

\begin{cor} \label{large}
There exist $3_c$-uniform $(n,k)$-frames for arbitrarily large values of $n$.
\end{cor}

The proof of Corollary \ref{large} follows from Proposition 3.6
in \cite{DHS}. While Theorem \ref{bingo} shows that arbitrarily large non-trivial $3_c$-uniform frames exist, there is still the question of ``Are there non-trivial $3_c$-uniform frames which come from
Seidel matrices with entries other than $i$ or $-i$?''. Theorem 
\ref{complex3unif} answers this question. Furthermore, it distinguishes
the complex case from the real case, and is the complex analog of Part (3) of Proposition \ref{threecorollaries}. 

\begin{thm}\label{complex3unif}
The trivial
$2_c$-uniform frames, corresponding
to $k=1$ or $k=n-1$, are $3_c$-uniform. In addition, $F$ is
a non-trivial $3_c$-uniform frame if and only if $F$ is a $2_c$-uniform frame arising from a real skew-symmetric matrix
$A$ 
with two distinct eigenvalues and entries $a_{i,j}=\pm 1$ when $i\ne j$ and $0$ otherwise.
\end{thm}

The following two lemmas will be used to prove Theorem \ref{complex3unif}.

\begin{lem}\label{3charpol}
Suppose $1\leq \lambda \leq \gamma \leq 3$ be the largest roots of the polynomials $x^3-3x^2+2-2\cos(\alpha)$ and $x^3-3x^2+2-2\cos(\beta)$, respectively. Then $\alpha\geq\beta$. Furthermore, when $0\leq\alpha\leq\beta\leq\pi$, equality holds if and only if $\lambda=\gamma$.
\end{lem}

\begin{proof} 
By assumption, $$\lambda-3\leq \gamma-3$$ which gives $$\lambda^2(\lambda-3)\leq \gamma^2(\gamma-3).$$ Combining this with the polynomials we get $$2-2\cos(\alpha)\leq 2-2\cos(\beta),$$ so $\cos(\alpha)\leq \cos(\beta)$ and $\alpha\geq\beta$.
\end{proof}

\begin{lem} \label{oneomega}
Suppose $F$ is a $3_c$-uniform frame with corresponding Seidel matrix $Q$. Then
the entries $q_{ij}$ of $Q$ are of the form $\omega$ or $\bar{\omega}$ when $1<i$, $1<j$ and $i\ne j$, for some fixed complex number $\omega$ with modulus $1$.
\end{lem}

\begin{proof}
Let $M$ and $N$ be two $3\times 3$ compressions of the Seidel matrix $Q$ corresponding to the $3$-uniform frame $F$. Since conjugating by an invertible matrix preserves eigenvalues, we can change $M$ and $N$ to be written as
$$M=\begin{bmatrix}
0 & 1 & 1 \\
1 & 0 & \alpha \\
1 & \bar{\alpha} &0
\end{bmatrix}$$ 
and
$$N=\begin{bmatrix}
0 & 1 & 1 \\
1 & 0 & \beta \\
1 & \bar{\beta} &0
\end{bmatrix}$$
where $\alpha$ and $\beta$ are complex numbers with modulus $1$. The characteristic polynomials of $M+I_3$ and $N+I_3$ are $x^3-3x^2+2-2\cos(\alpha)$ and  $x^3-3x^2+2-2\cos(\beta)$. Polynomials of this form are discussed in Proposition \ref{3charpol}. Since the norms of all $3\times 3$ compressions are equal, $\alpha$ and $\beta$ must be equal or conjugates.

The reverse direction is clear.
\end{proof}

\begin{proof} 
[of Thm. \ref{complex3unif}] In \cite{BP}, they observe that the 
trivial real $2_c$-uniform frames are $3_c$-uniform.

Without loss of generality assume that the Seidel matrix $Q$
associated with $F$ is in standard form. In the complex case, if $F$ 
is $3_c$-uniform, then Lemma \ref{oneomega} forces the off diagonal entries
of the $(n-1)\times (n-1)$ compression formed by removing the first row and column of 
$Q$ to be either of the form $\omega$ or $\bar{\omega}$ where $| \omega |=1$.

Suppose that  $i,j>1$, $i \not= j$, and the  $(i,j)$-entry of 
$Q$ is $\omega$. Using the fact that $Q^2=(n-1)I+\mu Q$ it follows that
$$
\mu \omega = m_1+m_2\omega^2+m_3\bar{\omega}^2
$$
where $m_1,m_2,$ and $m_3$ are positive integers. If $m_1=m_2$, then
$m_12\mathrm{Re}(\omega) +\mu=m_3\bar{\omega}^3$ which forces $\mu$ to be complex.
If $m_1 > m_2$, then
\begin{align} \label{dagger}
m_22\mathrm{Re}(\omega) + \mu &= (m_1-m_2)\bar{\omega}+m_3\bar{\omega}^3.
\end{align}
Clearly if $m_1-m_2 \not= m_3$, the right-hand side of (\ref{dagger}) is complex. 
On the other hand if $m_1-m_2=m_3$ and $\omega=\mathrm{e}^{i\theta}$ for 
some $0 \leq \theta < 2\pi$, then $\mathrm{e}^{i\theta}+\mathrm{e}^{i3\theta}$ must
be a real number. But this means that $\sin (\theta)+ \sin (3\theta)=0$ which occurs if and only if $\omega$ is a fourth root
of unity as desired.
\end{proof}

\begin{thm}\label{fouruniform}
The only nontrivial $4_c$-uniform frames are the ones in the equivalence class given by the $4\times 4$ Seidel matrix in Example \ref{3unif} previously mentioned.
\end{thm}

\begin{proof}
By Theorem \ref{complex3unif} we know that a nontrivial $3_c$-uniform frame corresponds to Seidel matrix $Q$ with entries
$$Q_{jm}=
\begin{cases}
0 & \text{if $m=j$,} \\
1 & \text{if $m\neq j$ and $m=1$ or $j=1$,} \\
\pm i & \text{otherwise}.
\end{cases}
$$ The proof of Theorem \ref{complex3unif} shows that the sum of the entries in each row and column of $q$, other than the first, is $1$.

Suppose $Q$ is an $n\times n$ matrix, then we use $Q$ to describe an edge coloring of the complete graph $K_{n-1}$. Label the vertices by the integers $2,\dots,n$. Color the edge from vertex $j$ to vertex $m$ red if $q_{jm}=-i$ and blue otherwise. It is well known, see \cite{CH,GRS}, that the Ramsey number $\operatorname{r}(3,3)=6$. With our interpretation of $Q$ giving a coloring, when $n\geq 7$ our coloring of $K_{n-1}$ contains a monochromatic triangle. The labels of the vertices of this triangle along with $1$ give us a $4\times 4$ compression of $Q$ of the form
$$
\begin{bmatrix}
0 & 1 & 1 & 1 \\
1 & 0 & i & i \\
1 & -i & 0 & i \\
1 & -i & -i & 0 \\
\end{bmatrix}
\text{ or }
\begin{bmatrix}
0 & 1 & 1 & 1 \\
1 & 0 & -i & -i \\
1 & i & 0 & -i \\
1 & i & i & 0 \\
\end{bmatrix}.
$$

Without loss of generality, assume that the first of the possible compressions above is the top left corner of $Q$. Since the row sums of $Q$ are $1$, there is a column of $Q$ such that $q_{j2}=-i$ and $q_{j3}=i$. With this, the $4\times 4$ compression using the rows and columns $\{1,2,3,j\}$ has the form
$$
\begin{bmatrix}
0 & 1 & 1 & 1 \\
1 & 0 & i & -i \\
1 & -i & 0 & i \\
1 & i & -i & 0 \\
\end{bmatrix}.
$$
These two compressions have different norms, so $Q$ is not $4$-uniform. A similar argument works for the other possible compression above.

The cases where $n<7$ have been checked computationally.
\end{proof}

In \cite{BP}, the authors showed that the only real $3_c$-uniform frames are the trivial $(n,n-1)$ and $(n,1)$ frames. Theorem \ref{complex3unif} extends this classification of $3_c$-uniform frames to the complex case. In addition to the real $3_c$-uniform frames, we add a new class, in particular the frames derived from real skew-symmetric matrices with exactly two eigenvalues. Theorem \ref{fouruniform} takes this classification one step farther to show that the only real or complex $4_c$-uniform frames are the trivial frames and one more, Example \ref{3unif}, which is $4_c$-uniform for the trivial reason that it has only one $4\times 4$ compression.

We end the paper by interpreting these results geometrically. A 
uniform $(n,k)$-frame yields a set of $n$-vectors in $\mathbb{R}^k$ (or 
$\mathbb{C}^k$) which have equal lengths. Another way to interpret $2$-uniform 
$(n,k)$-frames (or equivalently ETFs) is that the area of the parallelogram formed by
any two distinct vectors from such a frame is a constant. Intuitively, it would seem that
the  volume of the parallelepiped formed by choosing any three distinct vectors from a $2$-uniform
$(n,k)$-frame should be a constant. However, this is not true in general. In the real case, this is true if and only if the 
frame is trivial \cite{BP}, i.e., either an $(n,1)$ or $(n,n-1)$ frame. Similarly,
we have proven that in the complex case, the volume of the parallelepiped formed by choosing any three distinct vectors from a $2$-uniform
$(n,k)$-frame is a constant if and only if the Seidel matrix associated with
the frame comes from a real skew-symmetric with exactly two eigenvalues.

\end{document}